\documentclass{amsart}
\usepackage{amssymb}
\usepackage{graphics}
\usepackage{amsmath} 
\usepackage{tikz}
\usepackage{graphicx,subcaption}
\usetikzlibrary {positioning}

\theoremstyle{plain}

\def    \ut     {\Tilde{u}}

\def    \R  {{\Bbb R}}
\def    \Z  {{\Bbb Z}}

\def    \CP {{\Bbb {CP}}}

\def    \C  {{\Bbb C}}
\def    \N  {{\Bbb N}}

\def    \index  {{\operatorname{index}}}
\def    \Tilde  {\widetilde}

\def    \Gt     {\Tilde{G}}
\newtheorem*{theorem 1}{Theorem 1}
\def    \Tilde  {\widetilde}

\numberwithin{equation}{section}

\newtheorem{proposition}[equation]{Proposition}
\newtheorem{theorem}[equation]{Theorem}
\Huge
\newtheorem{lemma}[equation]{Lemma}
\newtheorem{corollary}[equation]{Corollary}
\theoremstyle{definition}

\newtheorem{example}[equation]{Example}

\begin{document}
\title[Hamiltonian $S^1$-manifolds in dimension 6]{Revisit Hamiltonian $S^1$-manifolds of dimension 6 with 4 fixed points}

\author{Hui Li} 
\address{School of mathematical Sciences\\
        Soochow University\\
        Suzhou, 215006, China.}
       \email{hui.li@suda.edu.cn}

\thanks{2020 classification. 53D05, 53D20, 55N25, 57R20}

\begin{abstract}
If the circle acts in a Hamiltonian way on a compact symplectic manifold of dimension $2n$, then there are at least $n+1$ fixed points. The case that there are exactly $n+1$ isolated fixed points has its importance due to various reasons.
Besides dimension 2 with 2 fixed points, and dimension 4 with 3 fixed points, which are known, the next interesting case is dimension 6 with 4 fixed points, for which the integral cohomology ring and the total Chern class of the manifold, and the sets of weights of the circle action at the fixed points are classified by Tolman. In this note, we use a new different argument to prove Tolman's results for the  dimension 6 with 4 fixed points case. We observe that the integral  cohomology ring of the manifold has a nice basis in terms of the moment map values of the fixed points, and the largest weight between two fixed points is nicely related to the first Chern class of the manifold. 
We will use these ingredients to determine the sets of weights of the circle action at the fixed points,
and moreover to determine the global invariants the integral cohomology ring and total Chern class of the manifold. The idea allows a direct approach of the problem, and the argument is short and easy to follow.  
\end{abstract}

\maketitle
\section{introduction}
If the circle acts in a Hamiltonian way on a compact $2n$-dimensional symplectic manifold $(M, \omega)$ with moment map $\phi$, then the fixed point set $M^{S^1}$ contains at least $n+1$ points. If $M^{S^1}$ consists of exactly $n+1$ points, we call the action has {\it minimal} isolated fixed points. When $\dim(M)=2$ with 2 fixed points, then $M$ can only be $S^2$. When 
$\dim(M)=4$ with 3 fixed points, then $M$ is $S^1$-equivariantly symplectomorphic to $\CP^2$ (see e.g. \cite{Ka} or  \cite{J1}).
Tolman studies the case when $\dim(M)=6$ and $M^{S^1}$ consists of 4 points (\cite{T}). 
Part of the motivation of her work comes from the classical Petrie's conjecture in the symplectic category.
She classifies the integral cohomology ring and the total Chern class of the manifold, and the sets of weights of the circle action at the fixed points. After the work in dimension 6, it comes the work 
in dimension 8 (\cite{GS, JT}),  in dimension 10 (\cite{L2}), and in dimension $2n$ with large first Chern classes (\cite{L1}). Now, for the purpose of other studies for dimension 6, we try to look at 
the case of dimension 6 with 4 fixed points again, in a point of view that this note presents, and reprove Tolman's results. Note that by the fact that the moment map is a perfect Morse function,  in dimension 6, having 4 fixed points is the same as restricting the second Betti number $b_2(M)=1$, since $b_4(M)=1$ by Poincar\'e duality and $b_0(M)=b_6(M)=1$ naturally hold. Tolman also considers nonisolated fixed points for the case $b_2(M)=1$. In this note, we restrict to isolated fixed points.

  Now we concentrate on $\dim(M)=6$ and $M^{S^1}$ consists of 4 points, i.e., $M^{S^1}=\{P_0, P_1, P_2, P_3\}$. At each $P_i$, a neighborhood of $P_i$ in $M$ is $S^1$-equivariantly diffeomorphic to a neighborhood of the tangent space $T_{P_i} M\approx\C^3$ on which $S^1$ acts linearly by 
$$\lambda \cdot (z_1, z_2, z_3) = (\lambda^{w_{i1}}z_1, \lambda^{w_{i2}}z_2, \lambda^{w_{i3}}z_3),$$ 
where $\lambda\in S^1$ and $w_{ik}\in\Z$ for $k=1, 2, 3$. The integers $w_{ik}$'s are well defined, and are called the {\it  weights} of the $S^1$-action at $P_i$.
Recall that for the Morse function the moment map $\phi$, the {\it Morse index of $P_i$} is twice the number of negative weights at $P_i$. 
By \cite[Lemma 3.3]{L1}, we can label the fixed points so that $P_i$ has Morse index $2i$ for 
$\phi$ with $i=0, 1, 2, 3$, and
\begin{equation}\label{order}
\phi(P_0) < \phi(P_1) < \phi(P_2) < \phi(P_3),
\end{equation}
moreover, $H^{2i}(M; \Z) =\Z$ with $i=0, 1, 2, 3$.
If the cohomology class $[\omega]$ represented by the symplectic form is integral, by Lemma~\ref{k|}, we have 
$$\phi(P_i)-\phi(P_j)\in \Z \,\,\,\mbox{for all $i, j$}.$$
 We observe that, for a Hamiltonian $S^1$-action on a compact symplectic manifold with {\it minimal} isolated fixed points, the equivariant and ordinary cohomology rings of the manifold have nice simple basis in terms of the moment map values of the fixed points. We also observe that, for a Hamiltonian $S^1$-action on a compact symplectic manifold $M$ with isolated fixed points, if in particular $H^2(M; \Z)=\Z$, then the largest weight between two fixed points is nicely related to the first Chern class of the manifold. For $\dim(M)=6$, if we use these ingredients  in the determination of the weights at the fixed points, then we find that the weights can be determined in a straightforward way. Moreover, we use the above mentioned basis of equivariant and ordinary cohomology rings of the manifold in terms of the moment map values of the fixed points to determine the global invariants, the integral cohomology ring and the total Chern class of the manifold. Tolman uses a different basis in \cite{T}. Our proofs lie in the use of Lemmas~\ref{ring}, \ref{c1}, \ref{ij}, and \ref{extension}. These observations show the strong link between the local data --- weights at the fixed points and moment map values of the fixed points, and the global intrinsic invariants --- the integral cohomology ring and the Chern classes of the manifold.

For a larger category, namely circle actions on a compact almost complex manifold $M$ of dimension 6 with 4 fixed points,  Ahara \cite{A}, and Jang \cite{J2} classify the sets of weights at the fixed points, where Ahara restricts to consider the case that Todd$(M)=1$ and $c_1^3[M]\neq 0$.
Here, for Hamiltonian circle actions on symplectic manifolds, besides the sets of weights at the fixed points, we also determine the two global  invariants, the integral cohomology ring and the total Chern class of the manifold. This reveals the close relationships between the local data related to the circle action and the global geometry and topology of the manifold. 

 Throughout the paper, a {\it Hamiltonian $S^1$-manifold}  means a symplectic manifold endowed with a Hamiltonian $S^1$-action. The results are stated as follows.
\begin{theorem 1}\label{thm}
Let $(M, \omega)$ be a compact effective Hamiltonian $S^1$-manifold of dimension 6 with 4 fixed points $P_0$, $P_1$, $P_2$, and $P_3$ and moment map $\phi$. Then the integral cohomology ring $H^*(M; \Z)$, the total Chern class $c(M)$, and the sets of weights at the fixed points have the  following 4 types. 
\begin{itemize}
\item [(1a)] $H^*(M; \Z)=\Z[x]/x^4$,  $c(M)= (1+x)^4$.
$$P_0\colon \{a, a+b, a+b+c\},\,\,\,  P_3\colon \{-a-b-c, -c-b, -c\},$$
                  $$P_1\colon \{-a, b, b+c\}, \,\,\, P_2\colon \{-a-b, -b, c\}.$$
In this case $\phi(P_1)-\phi(P_0)=a$, $\phi(P_2)-\phi(P_1)=b$ and $\phi(P_3)-\phi(P_2)=c$ if $[\omega]$ is primitive integral.
These data are the same as those of $\CP^3$ with a standard circle action.
\item [(1b)] $H^*(M; \Z)=\Z[x, y]/(x^2-2y, y^2)$, $c(M)= 1 + 3x + 8y + 4xy$.
 $$P_0\colon \Big\{a, a+b, \frac{1}{2}(2a+b)\Big\}, \,\,\, P_3\colon \Big\{-\frac{1}{2}(2a+b), -a-b, -a\Big\},$$
$$P_1\colon \Big\{-a, \frac{1}{2}b, a+b\Big\},\,\,\, P_2\colon \Big\{-a-b, -\frac{1}{2}b, a\Big\}.$$
In this case $\phi(P_1)-\phi(P_0)=\phi(P_3)-\phi(P_2)=a$ and $\phi(P_2)-\phi(P_1)=b$ if $[\omega]$
is primitive integral, where $b$ is even.
These data are the same as those of  $\Tilde G_2(\R^5)$ with a standard circle action.
\item [(2a)] $H^*(M; \Z)=\Z[x, y]/(x^2-5y, y^2)$, $c(M)= 1 + 2x + 12y + 4xy$.
$$P_0\colon \{1, 2, 3\},\,\,\, P_3\colon \{-1, -2, -3\},$$
 $$P_1\colon \{-1, 4, 1\},\,\,\, P_2\colon \{-1, -4, 1\}.$$
 In this case $\phi(P_1)-\phi(P_0)=\phi(P_3)-\phi(P_2)=1$ and $\phi(P_2)-\phi(P_1)=4$ if $[\omega]$ is primitive integral. These data are the same as those of  $V_5$ with a circle action.
\item [(2b)] $H^*(M; \Z)=\Z[x, y]/(x^2-22y, y^2)$, $c(M)= 1 + x + 24y + 4xy$.
$$P_0\colon \{1, 2, 3\}, \,\,\, P_3\colon \{-1, -2, -3\},$$
$$P_1\colon \{-1, 5, 1\}, \,\,\, P_2\colon \{-1, -5, 1\}.$$ In this case $\phi(P_1)-\phi(P_0)=\phi(P_3)-\phi(P_2)=1$ and $\phi(P_2)-\phi(P_1)=10$ if 
$[\omega]$ is primitive integral. These data are the same as those of  $V_{22}$ with a circle action.
\end{itemize}
Here, $\deg(x) =2$ and $\deg(y)=4$.
\end{theorem 1}

Theorem~\ref{thm} follows from Theorems~\ref{thmring}, \ref{Chern}, \ref{1ab}, and \ref{2ab},
Proposition~\ref{graph}, and Lemma~\ref{(3)}. 

For self completeness, we now give the examples $\CP^3$ and $\Tilde G_2(\R^5)$, corresponding to (1a) and (1b).

 \begin{example}\label{CP3}
Consider $\CP^3$ with the $S^1$-action:
$$\lambda\cdot [z_0, z_1, z_2, z_3] = [z_0, \lambda^a z_1,\lambda^{a+b}z_2,\lambda^{a+b+c}z_3],$$
where $a, b, c\in\N$.
This action has $4$ fixed points, $P_0 = [1, 0, 0, 0]$, $P_1 = [0, 1, 0, 0]$, $P_2 = [0, 0, 1, 0]$ 
and $P_3 = [0, 0, 0, 1]$, respectively with  moment map values $0$, $a$, $a+b$ and $a+b+c$. 
\end{example}
\begin{example}\label{grass}
Let $\Gt_2(\R^5)$ be the  Grassmannian of oriented $2$-planes in $\R^5$, and let the  $S^1$-action on $\Gt_2(\R^5)$ be induced from the $S^1$-action on $\R^5$  
given by
$$\lambda\cdot (t, z_0, z_1) = (t, \lambda^{a+\frac{b}{2}}z_0, \lambda^{\frac{b}{2}}z_1),$$
where $a\in\N$ and $b\in 2\N$. This action has $4$ fixed points, $P_0$,  $P_1$, $P_2$, and $P_3$,  where  $P_0$ and $P_3$ are the plane $(0, z_0, 0)$  with two different orientations, and $P_1$ and $P_2$ are the 
2-plane $(0, 0, z_1)$  with two different orientations. 
The moment map values of the fixed points are respectively
$-a-\frac{b}{2}$, $-\frac{b}{2}$, $\frac{b}{2}$, and $a +\frac{b}{2}$.
\end{example}

 Cases (2a) and (2b) in the theorem exist by McDuff --- in \cite{M}, she constructs them and shows that they are unique up to $S^1$-equivariant symplectomorphism. We take the names $V_5$ and $V_{22}$ from \cite{M}.

\section{preliminaries}
In this part, we summarize the main tools we will use in our proof.

Let $M$ be a smooth
 $S^1$-manifold. The {\bf equivariant cohomology} of $M$ in a coefficient ring $R$ is
 $H^*_{S^1}(M; R) = H^*(S^{\infty}\times_{S^1} M; R)$, where
 $S^1$ acts on $S^{\infty}$ freely. If $p$ is a point, then $H^*_{S^1}(p; R)= H^*(\CP^{\infty}; R)=R[t]$, where $t\in H^2(\CP^{\infty}; R)$ is a generator.
 If $S^1$ acts on $M$ trivially, then $H^*_{S^1}(M; R)= H^*(M; R)\otimes R[t] =  H^*(M; R)[t]$. The projection map $\pi\colon S^{\infty}\times_{S^1} M\to \CP^{\infty}$ induces a pull back map
$\pi^*\colon H^*(\CP^{\infty}) \to H^*_{S^1}(M)$,
so that $H^*_{S^1}(M)$ is an $H^*(\CP^{\infty})$-module.

First we have an equivariant extension of the symplectic class $[\omega]$ for a Hamiltonian $S^1$-manifold:

\begin{lemma}\cite{LT}\label{ut}
Let the circle act  on a compact symplectic manifold $(M,\omega)$
with moment map $\phi \colon M \to \R$.  Then there exists $\ut =[\omega -\phi t]\in H_{S^1}^2(M;\R)$ such that for any fixed point set component $F$,
$$\ut|_F =[\omega|_F]  - \phi(F)t.$$
If $[\omega]$ is an integral class, then $\ut$ is an integral class.  
\end{lemma}

For an $S^1$-manifold $M$,  when there exists a finite stabilizer group $\Z_k\subset S^1$ with $k > 1$, the set of points, $M^{\Z_k}$, which is pointwise fixed by $\Z_k$ but not pointwise fixed by $S^1$, is called a  {\bf $\Z_k$-isotropy submanifold}.

\begin{lemma}\label{k|}\cite{L1}
Let the circle act on a compact symplectic manifold
$(M, \omega)$ with moment map $\phi\colon M\to\R$. Assume  $[\omega]$ is an integral class. Then for any two fixed point set components $F$ and $F'$, $\phi(F) - \phi(F') \in \Z$.  If $\Z_k$ is the stabilizer group of some point on $M$,  then for any two fixed point set components $F$ and
$F'$ on the same connected component of $M^{\Z_k}$, we have $k\,|\left(\phi(F') - \phi(F)\right)$.
\end{lemma}

For a compact Hamiltonian $S^1$-manifold of dimension $2n$ with $n+1$ fixed points
$M^{S^1}=\{P_0,  P_1, \cdots, P_n\}$ and moment map $\phi$, we can label the fixed points so that  the following inequality holds (\cite{T, L1}):
$$\phi(P_0) < \phi(P_1) < \cdots < \phi(P_n).$$
The equivariant and ordinary cohomology of such a space have the following basis:

\begin{proposition}\label{equibase}\cite{{K}, {LT}}
Let the circle act on a compact $2n$-dimensional symplectic manifold $(M, \omega)$ with moment map $\phi\colon M\to\R$. Assume $M^{S^1}=\{P_0,  P_1, \cdots, P_n\}$. Then
as an $H^*(\CP^{\infty}; \Z)$-module,  $H^*_{S^1}(M; \Z)$ has a basis 
$\big\{\Tilde\alpha_i\in H^{2i}_{S^1}(M; \Z)\,|\, 0\leq i\leq n\big\}$ such that
$$\Tilde\alpha_i|_{P_i} = \Lambda_i^- t^i, \,\,\mbox{and}\,\,\,\, \Tilde\alpha_i|_{P_j} =0,\,\,\forall\,\, j < i,$$
where $\Lambda_i^-$ is the product of the negative weights at $P_i$.
Moreover, $\big\{\alpha_i = r(\Tilde\alpha_i) \in H^{2i} (M; \Z)\,|\, 0\leq i\leq n\big\}$  is a basis for $H^*(M; \Z)$, where
$r\colon H^*_{S^1}(M; \Z)\to H^*(M; \Z)$ is the natural restriction map.
\end{proposition}

A direct corollary of Proposition~\ref{equibase} is:

\begin{corollary}\label{cor}
Let the circle act on a compact $2n$-dimensional symplectic manifold $(M, \omega)$ with moment map $\phi\colon M\to\R$. Assume $M^{S^1}=\{P_0,  P_1, \cdots, P_n\}$. If 
$\Tilde\alpha\in H^{2i}_{S^1}(M; \Z)$ is a class such that $\Tilde\alpha|_{P_j}=0$ for all $j<i$, then
$$\Tilde\alpha =  a_i \Tilde\alpha_i \,\,\,\mbox{for some $a_i\in\Z$}.$$
\end{corollary}

Using these results, we obtain a basis of the ring $H^*(M; \Z)$:
\begin{lemma}\cite{L2}\label{ring}
Let  $(M, \omega)$ be a compact $2n$-dimensional  Hamiltonian $S^1$-manifold with moment map $\phi\colon M\to\R$. 
Assume $M^{S^1} = \{P_0,  P_1, \cdots, P_n\}$ and $[\omega]$ is a primitive integral class. Then the integral cohomology ring $H^*(M; \Z)$ is generated by the following 
 $\alpha_i\in H^{2i}(M; \Z)$'s for $0\leq i\leq n$:
$$\alpha_i = \frac{\Lambda_i^-}{\prod_{j=0}^{i-1}\big(\phi(P_j)-\phi(P_i)\big)}[\omega]^i,$$
where $\Lambda_i^-$ is the product of the negative weights at $P_i$. 
\end{lemma}

In the current arXiv version of \cite{L2}, Lemma~\ref{ring} is stated for dimension 10. It in fact holds for any dimension $2n$. The proof goes as follows. Using the basis $\{\Tilde\alpha_i\,|\, 0\leq i\leq n\}$ of $H^*_{S^1}(M; \Z)$ as in Proposition~\ref{equibase},  by Corollary~\ref{cor}, we have
$$\prod_{j=0}^{i-1}\big(\ut + \phi(P_j)t\big) = a_i\Tilde\alpha_i \,\,\,\mbox{for some $a_i\in\Z$}.$$
Restricting it to $P_i$, we get
$$\prod_{j=0}^{i-1}\big(\phi(P_j)-\phi(P_i)\big) = a_i\Lambda_i^-.$$
Restricting it to ordinary cohomology, we get
$$[\omega]^i = a_i \alpha_i.$$
The 2 equalities above give the result for the generator $\alpha_i$. Then the Lemma follows from 
Proposition~\ref{equibase} on the ring $H^*(M; \Z)$.

Next, we give 2 ways of representing the first Chern class of the manifold.

\begin{lemma}\cite{L1}\label{c1}
Let the circle act on a connected compact symplectic manifold $(M, \omega)$ with moment map 
$\phi\colon M\to\R$. If $c_1(M)=k[\omega]$ with $k\in\R$, then for any two fixed point set components
$F$ and $F'$,  we have 
$$\Gamma_F - \Gamma_{F'} = k\big(\phi(F')-\phi(F) \big),$$
where $\Gamma_F$ and $\Gamma_{F'}$ are respectively the sums of the weights at $F$ and $F'$.
\end{lemma}
In Lemma~\ref{c1}, if $[\omega]$ is a primitive integral class and $c_1(M)=k[\omega]$, then $k\in\Z$
since $c_1(M)$ is an integral class. If $M$ admits a Hamiltonian circle action such that 
$\Gamma_F \neq \Gamma_{F'}$ for minimum $F$ and maximum $F'$, then the constant $k\neq 0$.  
\smallskip

In a symplectic $S^1$-manifold $(M, \omega)$ with isolated fixed points,
if $w>0$ is a weight at a fixed point $P$, $-w$ is a weight  at a fixed point $Q$, and $P$ and $Q$ are on the same connected component of
$M^{\Z_w}$, we say that {\bf $w$ is a weight from $P$ to $Q$}. 
When the signs of $w$ at $P$ and at $Q$ are clear, we will also say that {\bf there is a weight $w$ or $-w$ between $P$ and $Q$}, or {\bf $w$ or $-w$ is a weight between $P$ and $Q$}.

\begin{lemma}\cite{{L1}}\label{ij}
Let the circle act on a connected compact symplectic manifold $(M, \omega)$ with moment map $\phi\colon M\to\R$. Assume $M^{S^1}$ consists of isolated points. Let $P, Q\in M^{S^1}$ with $P\neq Q$, where $\index (P)=2i$ and $\index (Q) = 2j$ with $i \leq j$. Assume there is a weight $w >0$ from $P$ to $Q$,  $-w$ is not a weight at $P$, $-w$  has multiplicity $1$ at $Q$, and $w$ is the largest among the absolute values of all the weights at $P$ and $Q$. If $c_1(M) = k [\omega]$ with $k\in\R$,  then 
$$j-i +1 = k\frac{\phi(Q)-\phi(P)}{w}.$$
\end{lemma}
Note that in Lemma~\ref{ij}, if $[\omega]$ is an integral class, then $\frac{\phi(Q)-\phi(P)}{w}\in\Z$ since $w\,|\,\big(\phi(Q)-\phi(P)\big)$ by Lemma~\ref{k|}; if $[\omega]$ is primitive integral, then 
$k\in\Z$. So if $[\omega]$ is primitive integral, then both factors on the right hand side divide the left hand side.
\smallskip

In our proof of determining the sets of weights at the fixed points, we mainly use Lemmas~\ref{ring},
\ref{c1} and \ref{ij}, this is our new idea and method of proofs. In determining the integral cohomology ring and total Chern class of the manifold, we primarily use Proposition~\ref{equibase},  
Corollary~\ref{cor} and Lemma~\ref{extension}. Other things we will use are the  following results. 

\begin{lemma}\cite{JT}\label{JT}
Let the circle act on a closed $2n$-dimensional almost
complex manifold $M$ with isolated fixed points. Let $w$ be the smallest
positive weight that occurs at the fixed points on $M$. Then given any 
$k \in\{0, 1, . . . , n-1\}$, the number of times the weight $-w$ occurs at
fixed points of index $2k+2$ is equal to the number of times the 
weight $+w$ occurs at fixed points of index $2k$.
\end{lemma}

\begin{lemma} \cite{T}\label{mod}
Let the circle act on a compact symplectic manifold $(M, \omega)$. Let $P$ and $Q$ be fixed points which lie on the same connected component of $M^{\Z_k}$ for some $k > 1$. Then the weights of the $S^1$-action at $P$ and $Q$ are equal modulo $k$.
\end{lemma}

\section{determining the sets of weights at the fixed points}

In this part, we determine the sets of weights of the circle action at all the fixed points.

First, we look at the possible connections of the weights at the fixed points.

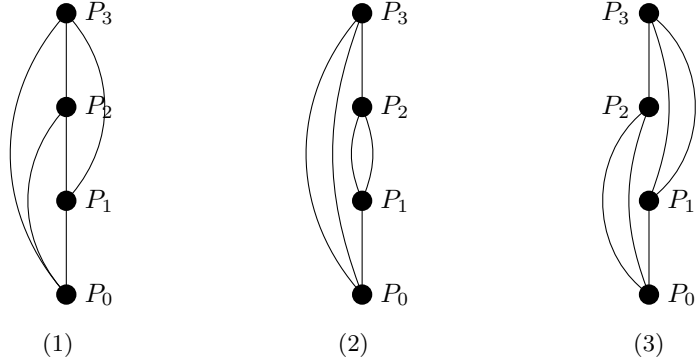
\begin{figure}
\centering
\begin{subfigure}[b][4.5cm][s]{.3\textwidth}
\centering
\vfill
\begin{tikzpicture}[colorstyle/.style={circle, draw=black,fill=black, thick, inner sep=0pt, minimum size=1 mm, outer sep=0pt}, scale=2]
\node (a) [colorstyle,scale=0.5,label={right:$P_0$}]{$P_0$};
\node (b) [above=of a] [colorstyle,scale=0.5,label={right:$P_1$}] {$P_1$};
\node (c) [above=of b] [colorstyle,scale=0.5,label={right:$P_2$}]{$P_2$};
\node (d) [above=of c] [colorstyle,scale=0.5,label={right:$P_3$}]{$P_3$};
\path (a) [-] edge node[right] {} (b);
\path (a) [-] [bend left =40] edge node [right] {} (c);
\path (a) [-] [bend left =40] edge node [right] {} (d);
\path (b) [-] edge node [right] {} (c);
\path (b) [-] [bend right =40]  edge node [right] {} (d);
\path (c) [-] edge node [right] {} (d);
\end{tikzpicture}
\vfill
\caption*{(1)}\label{fig1-1}
\vspace{\baselineskip}
\end{subfigure}
\begin{subfigure}[b][4.5cm][s]{.3\textwidth}
\centering
\vfill
\begin{tikzpicture}[colorstyle/.style={circle, draw=black,fill=black, thick, inner sep=0pt, minimum size=1 mm, outer sep=0pt}, scale=2]
\node (a) [colorstyle,scale=0.5,label={right:$P_0$}]{$P_0$};
\node (b) [above=of a] [colorstyle,scale=0.5,label={right:$P_1$}] {$P_1$};
\node (c) [above=of b] [colorstyle,scale=0.5,label={right:$P_2$}]{$P_2$};
\node (d) [above=of c] [colorstyle,scale=0.5,label={right:$P_3$}]{$P_3$};
\path (a) [-] [bend left =40] edge node[right] {} (d);
\path (a) [-] [bend left =20] edge node [right] {} (d);
\path (a) [-] edge node [right] {} (b);
\path (b) [-] [bend left=20] edge node [right] {} (c);
\path (b) [-] [bend right =20]  edge node [right] {} (c);
\path (c) [-] edge node [right] {} (d);
\end{tikzpicture}
\vfill
\caption*{(2)}\label{fig1-2}
\vspace{\baselineskip}
\end{subfigure}
\begin{subfigure}[b][4.5cm][s]{.3\textwidth}
\centering
\vfill
\begin{tikzpicture}[colorstyle/.style={circle, draw=black,fill=black, thick, inner sep=0pt, minimum size=1 mm, outer sep=0pt}, scale=2]
\node (a) [colorstyle,scale=0.5,label={right:$P_0$}]{$P_0$};
\node (b) [above=of a] [colorstyle,scale=0.5,label={right:$P_1$}] {$P_1$};
\node (c) [above=of b] [colorstyle,scale=0.5,label={left:$P_2$}]{$P_2$};
\node (d) [above=of c] [colorstyle,scale=0.5,label={left:$P_3$}]{$P_3$};
\path (a) [-] edge node[right] {} (b);
\path (a) [-] [bend left =20] edge node [pos=.7, left] {} (c);
\path (a) [-] [bend left =50]edge node [pos=.3, left] {} (c);
\path (b) [-] [bend right =20]  edge node [pos=.3, right] {} (d);
\path (b) [-] [bend right =50] edge node [pos=.7, right] {} (d);
\path (c) [-] edge node [right] {} (d);
\end{tikzpicture}
\vfill
\caption*{(3)}\label{fig1-3}
\vspace{\baselineskip}
\end{subfigure}\qquad
\caption{Three types for the weight relations at the fixed points}\label{fig1}
\end{figure}

\begin{proposition}\label{graph}
Let $(M, \omega)$ be a compact Hamiltonian $S^1$-manifold of dimension 6 with 4 fixed points $P_0$, $P_1$, $P_2$, and $P_3$. Then there are 3 types of weight relations as in Figure~\ref{fig1}.
\begin{enumerate}
\item  There is exactly one weight between any pair of fixed points.
\item  There is one weight between $P_0$ and $P_1$, one weight between $P_2$ and $P_3$, 2 weights between $P_1$ and $P_2$, and 2 weights between $P_0$ and $P_3$.
\item There is one weight between $P_0$ and $P_1$, one weight between $P_2$ and $P_3$,
2 weights between $P_0$ and $P_2$,  and 2 weights between $P_1$ and $P_3$.
\end{enumerate}
\end{proposition}

\begin{proof}
Let $\phi$ be the moment map. 
Let $-a$ be the negative weight at $P_1$. Then the connected component of $M^{\Z_a}$ containing $P_1$ is compact and symplectic (Hamiltonian),
and $P_1$ has index 2 in it, so there is a fixed point of index 0 below $P_1$ (relative to $\phi$), this fixed point can only be $P_0$. So there is a weight $a$ between $P_0$ and $P_1$. Similarly, using $-\phi$, we know that there is a weight between $P_2$ and $P_3$. There are 2 negative weights at $P_2$, for each negative weight  at $P_2$, we consider similarly the corresponding isotropy submanifold as above, we can see that there are 3 possibilities:
\begin{enumerate}
\item There is a weight between $P_0$ and $P_2$, and a weight between $P_1$ and $P_2$.
\item There are 2 weights between $P_1$ and $P_2$.
\item  There are 2 weights between $P_0$ and $P_2$.
\end{enumerate}
For each case, consider the fact that there are 3 weights at each fixed point  and the index of the fixed points, we see that 
the weight relation can only be as stated in the proposition.
\end{proof}

\begin{lemma}\label{(3)}
Let $(M, \omega)$ be a compact Hamiltonian $S^1$-manifold of dimension 6 with 4 fixed points $P_0$, $P_1$, $P_2$, and $P_3$.
Then (3) in Proposition~\ref{graph} is not possible.
\end{lemma}

\begin{proof}
Assume  (3) in Proposition~\ref{graph} holds. We will argue that it results in a contradiction.

Since $H^2(M; \Z)=\Z$ (see Lemma~\ref{ring}), with no loss of generality, assume  $[\omega]$ is primitive integral. Let $\phi$ be the moment map, by (\ref{order}), we let
$$\phi(P_1)-\phi(P_0)=a\in\N, \,\, \phi(P_2)-\phi(P_1)=b\in\N,\,\,  \phi(P_3)-\phi(P_2)=c\in\N.$$
By Lemma~\ref{ring}  for $\alpha_1=[\omega]$, we know that the weight between $P_0$ and $P_1$ is $a$; similarly, using $-\phi$, we get that the weight between $P_2$ and $P_3$ is $c$.
By (3) of Proposition~\ref{graph} and Lemma~\ref{k|}, we can write the sets of negative weights 
$P_i^-$ for $P_i$, where $i=1, 2, 3$, as follows.
$$P_1^-\colon \{-a\}.$$
$$P_2^-\colon \Big\{-\frac{a+b}{m_1},  -\frac{a+b}{m_2}\Big\}\,\,\,\mbox{for some $m_1, m_2\in\N$}.$$
$$P_3^-\colon \Big\{ -c, -\frac{b+c}{l_1}, -\frac{b+c}{l_2}\Big\} \,\,\,\mbox{for some $l_1, l_2\in\N$}.$$
 By Lemma~\ref{ring},
$$\alpha_1=[\omega], \,\, \alpha_2 =\frac{1}{m_1m_2}\frac{a+b}{b}[\omega]^2, \,\,\alpha_3=\frac{1}{l_1l_2}\frac{b+c}{a+b+c}[\omega]^3.$$ 
By Poincar\'e duality, $\alpha_1\alpha_2=\alpha_3$. So
$$l_1l_2(a+b)(a+b+c) = m_1m_2b(b+c).$$
Similarly, using $-\phi$, we get
$$m_1m_2(c+b)(a+b+c)=l_1l_2b(b+a).$$
The two equalities give a contradiction.
\end{proof}

\begin{theorem}\label{1ab}
Let $(M, \omega)$ be a compact Hamiltonian $S^1$-manifold of dimension 6 with 4 fixed points $P_0$, $P_1$, $P_2$, and $P_3$ and moment map $\phi$. Then the sets of weights at the fixed points in (1) of Proposition~\ref{graph} have 2 cases:
\begin{itemize}
\item [(1a)] $P_0\colon \{a, a+b, a+b+c\}$, $P_1\colon \{-a, b, b+c\}$, $P_2\colon \{-a-b, -b, c\}$, $P_3\colon \{-a-b-c, -c-b, -c\}$.
In this case $\phi(P_1)-\phi(P_0)=a$, $\phi(P_2)-\phi(P_1)=b$ and $\phi(P_3)-\phi(P_2)=c$ if $[\omega]$ is primitive integral.
This case is the same as those of a circle action on $\CP^3$. 
\item [(1b)] $P_0\colon \big\{a, a+b, \frac{2a+b}{2}\big\}$, $P_1\colon \big\{-a, \frac{b}{2}, a+b\big\}$, $P_2\colon \big\{-a-b, -\frac{b}{2}, a\big\}$, $P_3\colon \big\{-\frac{2a+b}{2}, -a-b, -a\big\}$.
In this case $\phi(P_1)-\phi(P_0)=\phi(P_3)-\phi(P_2)=a$ and $\phi(P_2)-\phi(P_1)=b$ if $[\omega]$
is primitive integral.
This case is the same as those of a circle action on $\Tilde G_2(\R^5)$.
\end{itemize}
\end{theorem}

\begin{proof}
Since $H^2(M; \Z)=\Z$, with no loss of generality, assume  $[\omega]$ is primitive integral,
and by (\ref{order}), let
$$\phi(P_1)-\phi(P_0)=a\in\N, \,\, \phi(P_2)-\phi(P_1)=b\in\N,\,\,  \phi(P_3)-\phi(P_2)=c\in\N.$$
By Lemma~\ref{ring} for $\alpha_1=[\omega]$, we get that the weight between $P_0$ and $P_1$ is $a$; similarly, using $-\phi$, we get that the weight between $P_2$ and $P_3$ is $c$.
By (1) of Proposition~\ref{graph} and Lemma~\ref{k|}, the sets of negative weights $P_i^-$ at $P_i$ for $i=1, 2, 3$  are:
$$P_1^-\colon \{-a\}.$$
$$P_2^-\colon \big\{-\frac{a+b}{l_1},  -\frac{b}{l_2}\big\}  \,\,\,\mbox{for some $l_1, l_2\in\N$}.$$
$$P_3^-\colon \big\{-\frac{a+b+c}{m_2}, -\frac{b+c}{m_1}, -c\big\}  \,\,\,\mbox{for some $m_1, m_2\in\N$}.$$
By Lemma~\ref{ring}, 
$$\alpha_1=[\omega], \,\, \alpha_2 = \frac{1}{l_1l_2}[\omega]^2,\,\, \alpha_3 = \frac{1}{m_1m_2}[\omega]^3.$$
By Poincar\'e duality, $\alpha_1\alpha_2=\alpha_3$. So
$$m_1m_2=l_1l_2.$$
Similarly, using $-\phi$, we get
$$l_1m_2 = m_1 l_2.$$
So 
$$m_1 = l_1 \,\,\,\mbox{and}\,\,\, m_2 = l_2.$$
Then we can write the sets of weights at the fixed points only using $m_1$ and $m_2$ as follows:
$$P_0\colon \Big\{a, \frac{a+b}{m_1}, \frac{a+b+c}{m_2}\Big\}, \,\,\,P_3\colon \Big\{-\frac{a+b+c}{m_2}, -\frac{b+c}{m_1}, -c\Big\},$$
$$P_1\colon \Big\{-a, \frac{b}{m_2}, \frac{b+c}{m_1}\Big\},\,\,\, P_2\colon \Big\{-\frac{a+b}{m_1}, -\frac{b}{m_2}, c\Big\}.$$
Since $H^2(M; \Z)=\Z$, we have $c_1(M)=k[\omega]$ for some $k\in\Z$. By Lemma~\ref{c1}, we have
$$\frac{\Gamma_0-\Gamma_3}{a+b+c} = \frac{\Gamma_1-\Gamma_2}{b},$$
where $\Gamma_i$ is the sum of weights at $P_i$,  from which we get
$$m_1 = 1.$$
By Lemma~\ref{c1}, we also have
$$\frac{\Gamma_0-\Gamma_1}{a} = \frac{\Gamma_2-\Gamma_3}{c},$$
from which we get
$$a=c\,\,\,\mbox{or}\,\,\, m_2 = m_1.$$
If $m_2=m_1=1$, then we get the sets of weights at the fixed points as in Case (1a).
Next, assume $a=c$ and $m_2\geq 2$. Then $a+b=\phi(P_2)-\phi(P_0)$ is the largest weight on $M$, is between $P_0$ and $P_2$ (and is between $P_1$ and $P_3$). By Lemma~\ref{ij}, we get
$$c_1(M) = 3[\omega].$$
Then  Lemma~\ref{c1} gives
$$\Gamma_0-\Gamma_1 = 3a,$$
from which we get
$$2+\frac{2}{m_2} = 3.$$
So $$m_2=2.$$
In this case we get the sets of weights at the fixed points as in Case (1b).
\end{proof}

\begin{theorem}\label{2ab}
Let $(M, \omega)$ be a compact  effective Hamiltonian $S^1$-manifold of dimension 6 with 4 fixed points $P_0$, $P_1$, $P_2$, and $P_3$ and moment map $\phi$. Then the sets of weights at the fixed points  in (2) of Proposition~\ref{graph} have 2 possibilities:
\begin{itemize}
\item [(2a)]  $P_0\colon \{1, 2, 3\}$, $P_1\colon \{-1, 1, 4\}$, $P_2\colon \{-1, -4, 1\}$, $P_3\colon \{-1, -2, -3\}$. In this case $\phi(P_1)-\phi(P_0)=\phi(P_3)-\phi(P_2)=1$ and $\phi(P_2)-\phi(P_1)=4$ if $[\omega]$ is primitive integral.
\item [(2b)]  $P_0\colon \{1, 2, 3\}$, $P_1\colon \{-1, 1, 5\}$, $P_2\colon \{-1, -5, 1\}$, $P_3\colon \{-1, -2, -3\}$. In this case $\phi(P_1)-\phi(P_0)=\phi(P_3)-\phi(P_2)=1$ and $\phi(P_2)-\phi(P_1)=10$ if 
$[\omega]$ is primitive integral.
\end{itemize}
\end{theorem}

\begin{proof}
Since $H^2(M; \Z)=\Z$, with no loss of generality, assume  $[\omega]$ is primitive integral,
and by (\ref{order}), let
$$\phi(P_1)-\phi(P_0)=a\in\N, \,\, \phi(P_2)-\phi(P_1)=b\in\N,\,\,  \phi(P_3)-\phi(P_2)=c\in\N.$$
By Lemma~\ref{ring} for $\alpha_1=[\omega]$, we know that the weight between $P_0$ and $P_1$ is $a$; similarly, the weight between $P_2$ and $P_3$ is $c$. By (2) of Proposition~\ref{graph} and Lemma~\ref{k|},
let the 2 weights between $P_0$ and $P_3$ be
$$\frac{a+b+c}{m_1}\,\,\,\mbox{and}\,\,\, \frac{a+b+c}{m_2},\,\,\,\mbox{with $m_1, m_2\in\N$},$$
and the 2 weights between $P_1$ and $P_2$ be
$$\frac{b}{l_1}\,\,\,\mbox{and}\,\,\, \frac{b}{l_2}, \,\,\,\mbox{with $l_1, l_2\in\N$}.$$
By Lemma~\ref{ring}, we get 
$$\alpha_1=[\omega], \,\,\, \alpha_2=\frac{1}{l_1l_2}\frac{b}{a+b}[\omega]^2,\,\,\, \alpha_3=\frac{1}{m_1m_2}\frac{a+b+c}{b+c}[\omega]^3.$$
 Poincar\'e duality $\alpha_1\alpha_2=\alpha_3$ yields
$$m_1m_2b(b+c)=l_1l_2(a+b)(a+b+c).$$
Similarly, using $-\phi$, we get
$$m_1m_2b(b+a)=l_1l_2(c+b)(a+b+c).$$
Hence 
$$a=c,$$
and
\begin{equation}\label{ringimp}
m_1m_2 b = l_1l_2(2a+b).
\end{equation}
Then we can write the sets of weights at the fixed points below:
$$P_0\colon \Big\{a, \frac{2a+b}{m_1},  \frac{2a+b}{m_2}\Big\},\,\,\,P_3\colon \Big\{ -\frac{2a+b}{m_1},  -\frac{2a+b}{m_2}, -a\Big\},$$
$$P_1\colon \Big\{-a, \frac{b}{l_1}, \frac{b}{l_2}\Big\},\,\,\, P_2\colon \Big\{-\frac{b}{l_1}, -\frac{b}{l_2}, a\Big\}.$$
If  $\gcd\Big(\frac{2a+b}{m_1}, \frac{2a+b}{m_2}\Big) = m > 1$, by effectiveness of the action at $P_0$, we have $m\nmid a$; let $C$ be the connected component of $M^{\Z_m}$ containing $P_0$, then the tangent space $T_{P_0}C$ cannot contain the complex line with weight $a$, so $C$ is a 4-dimensional symplectic manifold with only 2 fixed points $P_0$ and $P_3$ (it cannot contain $P_1$ or $P_2$), a contradiction. Hence
\begin{equation}\label{m12}
\gcd \Big(\frac{2a+b}{m_1}, \frac{2a+b}{m_2}\Big) = 1.
\end{equation}
Similarly, 
\begin{equation}\label{l12}
\gcd\Big(\frac{b}{l_1}, \frac{b}{l_2}\Big) = 1.
\end{equation}
If $\frac{2a+b}{m_1} =  \frac{2a+b}{m_2}=1$, then this contradicts to Lemma~\ref{JT} on the smallest weight on $M$. Together with
(\ref{m12}), with no loss of generality,  assume
\begin{equation}\label{m1<m2}
\frac{2a+b}{m_1} <  \frac{2a+b}{m_2}.
\end{equation}
With no loss of generality, assume
\begin{equation}\label{l1<l2}
 \frac{b}{l_1} \leq \frac{b}{l_2}.
\end{equation}
Since $H^2(M, \Z)=\Z$, we have $c_1(M)=k[\omega]$ for some $k\in\Z$. Then by Lemma~\ref{c1} for $P_1$ and $P_2$, and for $P_0$ and $P_3$, we get
\begin{equation}\label{12}
-2a+\frac{2b}{l_1}+\frac{2b}{l_2} = kb,
\end{equation}
and
\begin{equation}\label{03}
2a + 2\frac{2a+b}{m_1}+2\frac{2a+b}{m_2} = k(2a+b).
\end{equation}

 First we show that $a$ cannot be the largest weight on $M$. Assume instead that $a\geq \frac{b}{l_2}$ and $a\geq \frac{2a+b}{m_2}$. Then $a > 1$ by (\ref{m1<m2}). Then the smallest weight on $M$ is $\frac{2a+b}{m_1}$ or $\frac{b}{l_1}$.
If the smallest weight is  $\frac{2a+b}{m_1}$, then it contradicts to Lemma~\ref{JT}. If the smallest weight on $M$ is  $\frac{b}{l_1}$, by effectiveness of the action at $P_1$, we have $a > \frac{b}{l_1}$. By Lemma~\ref{mod}, we have 
$$\{\mbox{weights at $P_0$}\} =\{\mbox{weights at $P_1$}\} \mod a,$$
so we must have $\frac{b}{l_1} = \frac{2a+b}{m_1}$, again it contradicts to Lemma~\ref{JT}. 

  We next show that
\begin{equation}\label{a<}
 a < \frac{b}{l_2}.
\end{equation}
Assume instead that $a\geq \frac{b}{l_2}$. Since $a$ is not the largest weight on $M$, we have
\begin{equation}\label{a<m2}
a < \frac{2a+b}{m_2}.
\end{equation}
So $\frac{2a+b}{m_2}$ has multiplicity 1 at $P_0$ and  $-\frac{2a+b}{m_2}$ has multiplicity 1 at $P_3$, $\frac{2a+b}{m_2}$ is the largest weight at $P_0$ and $P_3$ and is between $P_0$ and $P_3$. If $m_2 = 1$, then by (\ref{m12}), 
$\frac{2a+b}{m_1} = 1$ is the smallest weight on $M$, is from $P_0$ to $P_3$, contradicting to Lemma~\ref{JT}. Hence $m_2\geq 2$. 
Then by Lemma~\ref{ij} for $P_0$ and $P_3$, we have 2 possibilities:
\begin{itemize}
\item [(i)]  $c_1(M) = 2[\omega]$ and $m_2=2$.
\item [(ii)]  $c_1(M) = [\omega]$  and $m_2=4$.
\end{itemize}
First consider (i). Then (\ref{12}) gives
$$ a-\frac{b}{l_2} =  \frac{b}{l_1} - b.$$
The left hand side is $\geq 0$ and the right hand side is $\leq 0$. If $l_1 > 1$, this is a contradiction.
If $l_1=1$, then $l_2=1$ by (\ref{l1<l2}) and then $a=b$.  By effectiveness of the action at $P_1$, we get $a=b=1$, then $\frac{2a+b}{m_2}$ with $m_2=2$ is not  an integer, a contradiction. 
Now consider (ii). Then  (\ref{12}) gives
$$ a-\frac{b}{l_2} =  \frac{b}{l_1} - \frac{b}{2}.$$
If $l_1 > 2$, then this is not possible. If $l_1=2$, then $a=\frac{b}{l_2}$. By (\ref{l1<l2}), 
$l_2=2$ or $l_2=1$. Then $a=b/2$ or $a=b$. When $l_1=l_2=2$ and $a=b/2$,
by effectiveness of the action at $P_1$, we get $a=1$ and $b=2$, then  (\ref{a<m2}) with $m_2=4$  give a contradiction. When $l_1=2$, $l_2=1$, and $a=b$, by effectiveness of the action at $P_1$, we get $a=b=2$, then  $\frac{2a+b}{m_2}$ with $m_2=4$ is not an integer, a contradiction. 

 The argument above shows that (\ref{a<}) holds.
So $\frac{b}{l_2} >1$. By (\ref{l12}) and (\ref{l1<l2}), we have
$$\frac{b}{l_1} < \frac{b}{l_2}.$$
Now $\frac{b}{l_2}$ is the largest weight at $P_1$ and $P_2$, and is between $P_1$ and $P_2$. By Lemma~\ref{mod},
$$\{\mbox{weights at $P_1$}\} =\{\mbox{weights at $P_2$}\} \mod \frac{b}{l_2}.$$
So either $ 2a=  \frac{b}{l_2}$ and $2\frac{b}{l_1}=\frac{b}{l_2}$, or $\pm a = \pm \frac{b}{l_1}$. In either case, we have
\begin{equation}\label{mod12}
 b=l_1 a.
\end{equation}
Now (\ref{mod12}) gives $a=\frac{b}{l_1}$. If $a\geq\frac{2a+b}{m_2}$, then the smallest weight on $M$ can only be
$\frac{2a+b}{m_1}$, contradicting to Lemma~\ref{JT}. So
$$ a < \frac{2a+b}{m_2}.$$
Since $\frac{b}{l_2}$ is the largest weight at $P_1$ and $P_2$, and it is between $P_1$ and $P_2$, by Lemma~\ref{ij}, we have 2 possibities:
\begin{itemize}
\item [(i')]  $c_1(M)=2[\omega]$ and $l_2=1$.
\item [(ii')]  $c_1(M)=[\omega]$ and $l_2 = 2$.
\end{itemize}

First consider Case (i').  Since  $\frac{2a+b}{m_2}$ is the largest weight at $P_0$ and $P_3$, and it is between $P_0$ and $P_3$, by Lemma~\ref{ij} for $P_0$ and $P_3$ and the fact $c_1(M)=2[\omega]$, we get
\begin{equation}\label{m2}
m_2 = 2.
\end{equation}
Equation (\ref{03}) now is 
\begin{equation}\label{03'}
a+ \frac{2a+b}{m_1} +\frac{2a+b}{m_2} = 2a +b.
\end{equation}
By (\ref{ringimp}), (\ref{mod12}), (\ref{m2}), (\ref{03'}), and the fact that $l_2=1$, we get
$$m_1 = 3 \,\,\,\mbox{and}\,\,\, l_1 = 4.$$
Now the set of weights at $P_1$ is $\{-a, a, 4a\}$; by effectiveness of the action, we get $a=1$. By (\ref{mod12}), we get
$$b=4.$$
The data above give us the sets of weights at the fixed points as in (2a).

Consider Case (ii'). Since  $\frac{2a+b}{m_2}$ is the largest weight at $P_0$ and $P_3$, and is between $P_0$ and $P_3$, by Lemma~\ref{ij} for $P_0$ and $P_3$ and the fact 
$c_1(M)=[\omega]$, we get
\begin{equation}\label{m2'}
m_2 = 4.
\end{equation}
Equation (\ref{03}) now is 
\begin{equation}\label{03''}
a+ \frac{2a+b}{m_1} +\frac{2a+b}{m_2} =\frac{2a +b}{2}.
\end{equation}
By (\ref{ringimp}), (\ref{mod12}),  (\ref{m2'}), (\ref{03''}),  and the fact that $l_2=2$, we get
$$m_1 = 6\,\,\,\mbox{and}\,\,\, l_1 = 10.$$
Then the set of weights at $P_1$  is $\{-a, a, 5a\}$; by effectiveness of the action, we get
$$a = 1.$$
Then by (\ref{mod12}),
$$b=10.$$
The data above give us the sets of weights at the fixed points as in (2b).
\end{proof}

\section{determining the ring $H^*(M; \Z)$ and the total Chern class c(M)}

In this section, we determine the integral cohomology ring $H^*(M; \Z)$ and the total Chern class
$c(M)$.

\begin{theorem}\label{thmring}
Let $(M, \omega)$ be a compact  effective Hamiltonian $S^1$-manifold of dimension 6 with 4 fixed points $P_0$, $P_1$, $P_2$, and $P_3$. Then $H^*(M; \Z)$ has the following 4 types, corresponding to
 (1a), (1b), (2a) and (2b) in Theorems~\ref{1ab} and \ref{2ab}. 
\begin{itemize}
\item [(1a)] $H^*(M; \Z)=\Z[x]/x^4$.
\item [(1b)] $H^*(M; \Z)=\Z[x, y]/(x^2-2y, y^2)$.
\item [(2a)] $H^*(M; \Z)=\Z[x, y]/(x^2-5y, y^2)$.
\item [(2b)] $H^*(M; \Z)=\Z[x, y]/(x^2-22y, y^2)$.
\end{itemize}
Here, $\deg(x) =2$ and $\deg(y)=4$.
\end{theorem}

\begin{proof}
Since $H^2(M; \Z)=\Z$, with no loss of generality, assume $[\omega]$ is primitive integral. Let $\phi$ be the moment map,
and by (\ref{order}) let 
$$\phi(P_1)-\phi(P_0)=a\in\N, \,\, \phi(P_2)-\phi(P_1)=b\in\N,\,\,  \phi(P_3)-\phi(P_2)=c\in\N.$$

For (1a), by (1a) in Theorem~\ref{1ab} on the set of weights at $P_3$, and by Lemma~\ref{ring}, we get that
$$\alpha_3 = [\omega]^3\in H^6(M; \Z)$$
is a generator. 
Since $\alpha_1 = [\omega]\in H^2(M; \Z)$ is a generator, by Poincar\'e duality, $\alpha_2=[\omega]^2\in H^4(M; \Z)$ is a generator. By Lemma~\ref{ring}, $H^*(M; \Z)$ is generated by 
$$1,\,\, [\omega],\,\, [\omega]^2,\,\, [\omega]^3.$$
Let $[\omega]=x$, we get the stated ring for (1a).

For (1b), by (1b) in Theorem~\ref{1ab} on the set of weights at $P_3$ with $a=c$, and by Lemma~\ref{ring}, we get that
$$\alpha_3 = \frac{1}{2}[\omega]^3\in H^6(M; \Z)$$
is a generator. Since $\alpha_1 = [\omega]\in H^2(M; \Z)$ is a generator, by Poincar\'e duality, $\alpha_2=\frac{1}{2}[\omega]^2\in H^4(M; \Z)$ is a generator. By Lemma~\ref{ring}, $H^*(M; \Z)$ is generated by 
$$1,\,\, [\omega],\,\, \frac{1}{2}[\omega]^2,\,\, \frac{1}{2}[\omega]^3.$$
Let  $[\omega]=x$ and $\frac{1}{2}[\omega]^2 =y$, then $x^2 = 2y$ and we get the stated ring for (1b).

Similarly, for (2a), by  (2a) in Theorem~\ref{2ab} on the sets of weights at the fixed points with $a=c=1$ and $b=4$, and by Lemma~\ref{ring}, we get that $H^*(M; \Z)$ is generated by 
$$1,\,\, [\omega],\,\, \frac{1}{5}[\omega]^2,\,\, \frac{1}{5}[\omega]^3.$$
Let  $[\omega]=x$ and $\frac{1}{5}[\omega]^2 =y$, then $x^2 = 5y$ and we get the stated ring for (2a).

Similarly, for (2b), by  (2b) in Theorem~\ref{2ab} on the sets of weights at the fixed points with $a=c=1$ and $b=10$, and by Lemma~\ref{ring}, we get that
 $H^*(M; \Z)$ is generated by 
$$1,\,\, [\omega],\,\, \frac{1}{22}[\omega]^2,\,\, \frac{1}{22}[\omega]^3.$$
Let  $[\omega]=x$ and $\frac{1}{22}[\omega]^2 =y$, then $x^2 = 22y$ and we get the stated ring for (2b).
\end{proof}

\begin{lemma}\label{extension}
Let $(M, \omega)$ be a compact  effective Hamiltonian $S^1$-manifold of dimension 6 with 4 fixed points $P_0$, $P_1$, $P_2$, and $P_3$, and with moment map $\phi$. 
Assume $[\omega]$ is primitive integral.
Then the equivariant cohomology $H_{S^1}^*(M; \Z)$ as an $H^*(\CP^{\infty}; \Z)$ module has the following 4 types of basis, corresponding to (1a), (1b), (2a) and (2b) in Theorems~\ref{1ab} and \ref{2ab}, where $\ut$ is as in Lemma~\ref{ut}.
\begin{itemize}
\item [(1a)]  $1$,  $\Tilde\alpha_1=\ut+\phi(P_0)t$,  $\Tilde\alpha_2=\prod_{j=0}^{1}\big(\ut +\phi(P_j)t\big)$, $\Tilde\alpha_3=\prod_{j=0}^{2}\big(\ut +\phi(P_j)t\big)$.
\item [(1b)]  $1$,  $\Tilde\alpha_1=\ut+\phi(P_0)t$,  $\Tilde\alpha_2=\frac{1}{2}\prod_{j=0}^{1}\big(\ut +\phi(P_j)t\big)$, $\Tilde\alpha_3=\frac{1}{2}\prod_{j=0}^{2}\big(\ut +\phi(P_j)t\big)$.
\item [(2a)] $1$,  $\Tilde\alpha_1=\ut+\phi(P_0)t$,  $\Tilde\alpha_2=\frac{1}{5}\prod_{j=0}^{1}\big(\ut +\phi(P_j)t\big)$, $\Tilde\alpha_3=\frac{1}{5}\prod_{j=0}^{2}\big(\ut +\phi(P_j)t\big)$.
\item [(2b)] $1$,  $\Tilde\alpha_1=\ut+\phi(P_0)t$,  $\Tilde\alpha_2=\frac{1}{22}\prod_{j=0}^{1}\big(\ut +\phi(P_j)t\big)$, $\Tilde\alpha_3=\frac{1}{22}\prod_{j=0}^{2}\big(\ut +\phi(P_j)t\big)$.
\end{itemize}
\end{lemma}

\begin{proof}
Let $\{\Tilde\alpha_i\,|\, 0\leq i\leq 3\}$ be the basis of $H^*_{S^1}(M; \Z)$ as in Proposition~\ref{equibase}. By Corollary~\ref{cor},
$$\prod_{j < i}(\ut +\phi(P_j)t) = a_i\Tilde\alpha_i\,\,\,\mbox{with $a_i\in\Z$}.$$
Restricting this to ordinary cohomology, we get
$$[\omega]^i = a_i\alpha_i.$$
By Theorem~\ref{thmring}, for each case (1a), (1b), (2a) and (2b), we know the relation of $\alpha_i$ with $[\omega]^i$, so we know $a_i$. Then plugging in the $a_i$ to the identity on 
$\Tilde\alpha_i$, we get the results.
\end{proof}

\begin{theorem}\label{Chern}
Let $(M, \omega)$ be a compact effective Hamiltonian $S^1$-manifold of dimension 6 with 4 fixed points $P_0$, $P_1$, $P_2$, and $P_3$. Then the total Chern class $c(M)$ has the following 4 types, corresponding to (1a), (1b), (2a) and (2b) in Theorems~\ref{1ab} and \ref{2ab}.
\begin{itemize}
\item [(1a)] $c(M)= (1+x)^4$.
\item [(1b)] $c(M)= 1 + 3x + 8y + 4xy$.
\item [(2a)] $c(M)= 1 + 2x + 12y + 4xy$.
\item [(2b)] $c(M)= 1 + x + 24y + 4xy$.
\end{itemize}
Here, $\deg(x) =2$ and $\deg(y)=4$.
\end{theorem}

\begin{proof}
Since $H^2(M; \Z)=\Z$, with no loss of generality, assume $[\omega]$ is primitive integral.
As before, we use $\Gamma_i$ to denote the sum of weights at $P_i$, and use $\Lambda_i^-$ to denote the product of the negative weights at $P_i$, for $i=0, 1, 2, 3$.

Using the basis  of $H^*_{S^1}(M; \Z)$  and the degree of the equivariant first Chern class $c_1^{S^1}(M)$, we can write 
\begin{equation}\label{eqc1}
c_1^{S^1}(M) = a_0t + a_1\Tilde\alpha_1, \,\,\,\mbox{where $a_0, a_1\in\Z$}.
\end{equation}  
Restricting (\ref{eqc1}) to $P_0$, we get
$$\Gamma_0 = a_0.$$
Restricting (\ref{eqc1}) to $P_1$, we get
$$\Gamma_1 = a_0 + a_1\Lambda_1^-.$$
Solving the two equalities for $a_0$ and $a_1$ and plugging them back to (\ref{eqc1}),  we get 
$$c_1^{S^1}(M) = \Gamma_0t+\frac{\Gamma_1-\Gamma_0}{\Lambda_1^-}\Tilde\alpha_1.$$
Restricting this equality to ordinary cohomology, we get
$$c_1(M) = \frac{\Gamma_1-\Gamma_0}{\Lambda_1^-}\alpha_1.$$
Similarly, for the equivariant second Chern class $c_2^{S^1}(M)$, we can write 
$$c_2^{S^1}(M) = b_0t^2 + b_1t\Tilde\alpha_1 + b_2\Tilde\alpha_2, \,\,\,\mbox{where $b_0, b_1, b_2\in\Z$}.$$ 
As for the case of $c_1^{S^1}(M)$, we respectively restrict this equality to $P_0$, $P_1$ and $P_2$, solve 3 equations for the constants $b_0$, $b_1$ and $b_2$, and plug them back to $c_2^{S^1}(M)$, we can get that
$$c_2^{S^1}(M) = (\sigma_2(P_0)) t^2 + \frac{\sigma_2(P_1)-\sigma_2(P_0)}{\Lambda_1^-}t\Tilde\alpha_1+ \frac{\sigma_2(P_2)-\sigma_2(P_0)-\frac{\sigma_2(P_1)-\sigma_2(P_0)}{\Lambda_1^-}\big(\Tilde\alpha_1(P_2)\big)}{\Lambda_2^-}\Tilde\alpha_2.$$ 
Here, $\sigma_2(P_i)$ is the degree 2 symmetric polynomial in the weights at $P_i$.
Restricting the last equality to ordinary cohomology, we get
$$c_2(M) = \frac{\sigma_2(P_2)-\sigma_2(P_0)-\frac{\sigma_2(P_1)-\sigma_2(P_0)}{\Lambda_1^-}\big(\Tilde\alpha_1(P_2)\big)}{\Lambda_2^-}\alpha_2.$$
Since the top Chern number is equal to the number of fixed points, we have
$$c_3(M) = 4\alpha_3.$$

For each case, using the basis of $H^*_{S^1}(M; \Z)$ as in Lemma~\ref{extension}, the sets of weights in Theorems~\ref{1ab} and \ref{2ab}, and the ring $H^*(M; \Z)$ in Theorem~\ref{thmring}, we can obtain the results by computing $c_1(M)$ and $c_2(M)$.  For example, for (1a), we get
 $$c(M)=1+4\alpha_1 + 6\alpha_2 + 4\alpha_3.$$
By  (1a)  of Theorem~\ref{thmring}, we have $\alpha_1=x$, $\alpha_2=x^2$, $\alpha_3=x^3$. So 
$c(M)$ for (1a) follows.
\end{proof}

\end{document}